\documentclass[a4paper,12pt]{article}
\usepackage[utf8]{inputenc}
   \usepackage{hyperref}
   \usepackage{amsmath}
   \usepackage{amsfonts}
   \usepackage{amssymb}
   \usepackage{mathrsfs}
   \usepackage{graphicx}
   \usepackage[abs]{overpic}
   \usepackage{float}
   \usepackage{cases}
\newtheorem{theorem}{Theorem}
\newtheorem{lemma}{Lemma}
\newtheorem{corollary}{Corollary}

\newenvironment{proof}{\smallskip\noindent{\bf Proof}\rm}
{\hfill $\Box$\medskip}
\newcommand{\be}{\begin{equation}}
\newcommand{\ee}{\end{equation}}
\newcommand{\ba}{\begin{array}}
\newcommand{\ea}{\end{array}}
\newcommand{\bea}{\begin{eqnarray*}}
\newcommand{\eea}{\end{eqnarray*}}
\newcommand{\bean}{\begin{eqnarray}}
\newcommand{\eean}{\end{eqnarray}}

\makeatletter \@addtoreset{equation}{section}

\makeatother

\begin{document}

\title{ Eigenvalue ratios for vibrating String equations with concave densities }

\author{Jihed Hedhly \thanks{Facult\'e des Sciences de Tunis, Universit\'{e} El-Manar,
 Laboratoire Equations aux D\'{e}riv\'{e}es Partielles,
, jihed.hedhly@fst.utm.tn}}

\date{}

 \maketitle

\begin{abstract} In this paper, we prove the optimal lower bound
$\frac{\lambda_n}{\lambda_m}\geq(\frac{n}{m})^2$ of vibrating string
$$-y''=\lambda\rho(x) y,$$ with Dirichlet boundary conditions for concave densities.
Our aproach is based on the method of Huang [Proc. AMS., 1999]. The
main argument is to restrict the two consecutive eigenfunction
$y_{n-1}$ and $y_n$ between two successive zeros of $y_{n-1}$.
%
 We also prove the same result for the Dirichlet
Sturm-Liouville problems.

\end{abstract}
~~~~~~~~~~
\maketitle
\section{ Introduction}
We consider the Sturm-Liouville equation acting on $[0,1]$
\begin{equation} \label{1.1S}
-(p(x)y^{\prime})^ {\prime }+q(x)y=\lambda \rho (x)y,
\end{equation}
 with Dirichlet boundary
conditions
\begin{equation}
y(0)=y(1)=0  \label{1.2},
\end{equation}
where $p>0$, $\rho>0$ and $q$ (may change sign) are continuous
coefficients on $[0,1]$. Here we limit ourselves to the case
$\rho>0$. The case $\rho<0$ has been considered for related problems
providing different results, we refer to pioneering works
\cite{AC1,AC2} and some refer therein.

As is well-known (see \cite{5}), there exist two countable sequences
of eigenvalues
$$\lambda
_{1}<\lambda _{2}<\dots<\lambda _{n}\dots\infty.$$
  The issues of optimal estimates for the eigenvalue ratios
$\frac{\lambda_n}{\lambda_m}$ have attracted a lot of attention
(cf.\cite{1,2',JJ,J,H99,H2016,4,MH,5',M'}) and references therein.
Ashbaugh and Benguria proved in  \cite{2'} that if $q\geq0$ and
$0<k\leq p\rho(x)\leq K$, then the eigenvalues of
\eqref{1.1S}-\eqref{1.2} satisfy
%
\begin{eqnarray*}
\frac{\lambda _{n}}{\lambda _{1}}\leq \frac{Kn^{2}}{k}.
\end{eqnarray*} 
They also established the following ratio estimate (of two arbitrary
eigenvalues) $$\frac{\lambda _{n}}{\lambda _{m}}\leq
\frac{Kn^2}{km^2},\quad n> m\geq1,$$ with $q\equiv0$ and $0<k\leq
p\rho(x)\leq K$. Later, Huang and Law \cite{4} extended the results
in \cite{2'} to more general boundary conditions. Racently, J.
Hedhly \cite{J}, showed that $$\frac{\lambda _{n}}{\lambda _{m}}\leq
\frac{n^2}{m^2},\quad n> m\geq1,$$ for single-barrier potential $q$
and single-well $p\rho$. He also established that the eigenvalues
for the string equation
\begin{equation} \label{1.1}
-y''=\lambda\rho(x) y,
\end{equation} with Dirichlet boundary
conditions \eqref{1.2} satisfy
$$\frac{\lambda _{n}}{\lambda _{m}}\leq
\frac{n^2}{m^2},\quad n> m\geq1,$$ with
single-well density $\rho$. 
  \\
 Recall that $f$ is a single-barrier (resp.
single-well) function on $[0, 1]$ if there is a point $x_{0} \in[0,
1]$ such that $f$ is increasing (resp. decreasing) on $[0, x_{0}]$
and decreasing (resp. increasing) on $[x_{0}, 1]$ (see  \cite{2}).

In this paper, we prove the optimal lower bound
$\frac{\lambda_n}{\lambda_m}\geq(\frac{n}{m})^2$
of\eqref{1.1}-\eqref{1.2} for concave density $\rho$. Our aproach is
based on the method of Huang [Proc. AMS., 1999]. The main argument
is to restrict the two consecutive eigenfunction $y_{n-1}$ and $y_n$
between two successive zeros of $y_{n-1}$, say $x_i$ and $x_{i+1}$.
We prove arguing as in $[x_i,x_{i+1}]$ that
$$\int_{x_i}^{x_{i+1}}x(y_{n-1}(x,\tau)^2-y_n(x,\tau)^2)\geq0.$$ We
also prove an result for the Dirichlet Sturm-Liouville problems
\eqref{1.1S}-\eqref{1.2}. More precisely, we show that
$\frac{\lambda_n}{\lambda_m}\geq(\frac{n}{m})^2$ with $q\equiv0$ and
$p\rho$ concave.
\section{Eigenvalue ratio for the vibrating String equations }
Denote by $u_n(x)$ be the $n-th$ eigenfunction of \eqref{1.1}
corresponding to $\lambda_n$, normalized so that
$$\int_0^1\rho(x) u_n^2(x)dx=1.$$

It is well known that the $u_n(x)$ has exactly $(n-1)$ zeros  in the
open interval $(0,1)$. The zeros of the $n-th$ and $(n +1)st$
eigenfunctions interlace, i.e. between any two successive zeros of
the $n-th$ eigenfunction lies a zero of the $(n+1)st$ eigenfunction.
\\ We denote by $(y_i)_i$ the zeros of $u_n$ and $(z_i)_i$ the zeros of
$u_{n-1}$, then in view of the comparison theorem (see
\cite[Chap.1]{5}), we have $y_i<z_i$. We may assume that $u_n(x)>0$
and $u_{n-1}(x)>0$ on $(0,y_1)$, then we have
$\frac{u_n(x)}{u_{n-1}(x)}$ is strictly decreasing on $(0, 1).$ In
deed,
$$(\frac{u_n(x)}{u_{n-1}(x)})'=\frac{u'_n(x)u_{n-1}(x)-u'_{n-1}(x)u_{n}(x)}{u^2_{n-1}(x)}
=\frac{w(x)}{u^2_{n-1}(x)}.$$ We find
$$w'(x)=u''_n(x)u_{n-1}(x)-u''_{n-1}(x)u_{n}(x)=(\lambda_{n-1}-\lambda_n)\rho(x)u_n(x)u_{n-1}(x),$$
this implies that $w(x) < 0$ on $(0, 1)$. Hence
$\frac{u_n(x)}{u_{n-1}(x)}$ is strictly decreasing on $(0, 1).$ From
this, there are points $x_i\in(y_i,z_i)$ such that
\begin{eqnarray*}  \left\{
\begin{array}{ll} u^2_n(x)>u^2_{n-1}(x),\ \ \ x\in(x_{2i},x_{2i+1}),\\~~\\
 u^2_n(x)<u^2_{n-1}(x),\ \ \ x\in(x_{2i+1},x_{2i+2}).
 \end{array}
 \right.
\end{eqnarray*}

Let $\rho(., \tau)$ is a one-parameter family of piecewise
continuous densities such that $\frac{\partial\rho(.,
\tau)}{\partial\tau}$ exists, and let $u_n(x,\tau)$ be the $n-th$
eigenfunction of \eqref{1.1} corresponding to $\lambda_n(\tau)$ of
the corresponding String equation \eqref{1.1} with $\rho=\rho(.,
\tau)$.
 From Keller in \cite{KL}, we get
$$\frac{d}{d\tau}\lambda_n(\tau)=-\lambda_n(\tau)
\int_0^1\frac{\partial\rho}{\partial \tau}(x,\tau)u_n^2(x,\tau)dx.$$
By straightforward computation that, yields
\begin{equation}\label{Ra}\frac{d}{d\tau}\Big[\frac{\lambda_n(\tau)}{\lambda_{m}(\tau)}\Big]=
\frac{\lambda_n(\tau)}{\lambda_{m}(\tau)}
\int_0^1\frac{\partial\rho}{\partial
\tau}(x,\tau)(u_{m}^2(x,\tau)-u_n^2(x,\tau))dx.\end{equation}

 We are now in position to state our main result.
\begin{theorem}\label{theo11}
Let $\rho$ a concave density on $[0,1]$. Then the eigenvalues of the
Dirichlet problem \eqref{1.1}-\eqref{1.2} satisfy
   \begin{equation}\label{RA}
   \frac{\lambda_{n}}{\lambda_{m}}\geq(\frac{n}{m})^2,
   \end{equation}
with equality if and only if $\rho$ is constant.
 \end{theorem}

In order to prove Theorem \ref{theo11} we need some preliminary
results, in particular the following result by Huang \cite{H99}.
\begin{lemma} \label{LM1}\cite{H99}
If $g$ is three times differentiable and $u$ satisfies
$$-y''=\lambda\rho(x) y,\ \ \ 0\leq x\leq1, \ \ \ y(0)=y(1)=0,$$\\ where $\rho$
is differentiable, then
$$g(1)y'(1)^2-g(0)y'(0)^2=\int_0^1\Big[2\lambda g'(x)\rho(x)+\lambda g(x)\rho'(x)+\frac{1}{2}g'''(x)\Big]y^2(x)dx.$$
\end{lemma}
\begin{lemma}\label{lm}
Consider the one-parameter family of linear densities $\rho(x, \tau)
= \tau x+b$, where $t > 0$ and $b$ is a positive constant. Let
$\lambda_n(\tau)$ be the $nth$ eigenvalue of \eqref{1.1}-\eqref{1.2}
with $\rho = \rho(x,\tau).$ Then the ratio
$\frac{\lambda_n(\tau)}{\lambda_{n-1}(\tau)}$ is a strictly
increasing function of t.
\end{lemma}
\begin{proof} From \eqref{Ra}
\begin{eqnarray*}&&\frac{d}{d\tau}\Big[\frac{\lambda_n(\tau)}{\lambda_{n-1}(\tau)}\Big]=
\frac{\lambda_n(\tau)}{\lambda_{n-1}(\tau)}
\int_0^1x(y_{n-1}^2(x,\tau)-y_n^2(x,\tau))dx.\end{eqnarray*} So, we
have to show that
\begin{eqnarray}\label{x}\ll x(\tau)\gg=\int_0^1x(y_{n-1}^2(x,\tau)-y_n^2(x,\tau))dx\geq0,
\end{eqnarray} for
all $\tau>0$.\\ Firstly notice that
\begin{eqnarray*}
&&\ll x(\tau)\gg =\int_0^1x(y_{n-1}^2(x,\tau)-y_n^2(x,\tau))dx
\cr&=&\sum_{i=1}^{n-1}\int_{z_i}^{z_{i+1}}x(y_{n-1}^2(x,\tau)-y_n^2(x,\tau))dx.
\end{eqnarray*}
To show \eqref{x}, it suffices to show that
$$\ll x(\tau)\gg_i=\int_{z_i}^{z_{i+1}}x(y_{n-1}^2(x,\tau)-y_n^2(x,\tau))d\tau\geq0$$

 Taking $g(x) = x$ in
Lemma \ref{LM1}, we get
$$  y'_n(z_{i+1},\tau)^2=\lambda_n\int_{z_i}^{z_{i+1}}(3\tau x+2b)+y_n^2(x,\tau))dx,$$
and, with $g(x)=x^2,$
$$y'_n(z_{i+1},\tau)^2=\lambda_n\int_{z_i}^{z_{i+1}}(5\tau x^2+4bx)+y_n^2(x,\tau))dx.$$
Therefore,
\begin{eqnarray}\label{g}
5t\int_{z_i}^{z_{i+1}}x^2y_n^2(x,\tau)dx=(3\tau-4b)\int_{z_i}^{z_{i+1}}x
y_n^2(x,\tau)dx+2ab\int_{z_i}^{z_{i+1}}y_n^2(x,\tau)dx.
\end{eqnarray}
By the normalization condition of $y $, we obtain
$$\int_0^1y_n^2(x,\tau)dx=\frac{1}{b}-\frac{\tau}{b}\int_0^1xy_n^2(x,\tau)dx,$$
yields
\begin{eqnarray*}
&&\int_0^1y_n^2(x,\tau)dx=\sum_{i=0}^{n}\int_{z_i}^{z_{i+1}}y_n^2(x,\tau)dx
\cr&=&\sum_{i=0}^{n}\Big[\frac{1}{nb}-
\frac{\tau}{b}\int_{z_i}^{z_{i+1}}xy_n^2(x,\tau)dx\Big].
\end{eqnarray*}
 Then
$$\int_{z_i}^{z_{i+1}}y_n^2(x,\tau)dx=\frac{1}{nb}-
\frac{\tau}{b}\int_{z_i}^{z_{i+1}}xy_n^2(x,\tau)dx.$$ Thus, from
\eqref{g}, we get
$$\int_{z_i}^{z_{i+1}}x^2y_n^2(x,\tau)dx=\frac{2}{5\tau}+\frac{\tau-4b}{b}
\int_{z_i}^{z_{i+1}}xy_n^2(x;\tau)dx.$$ From this, it follows that
\begin{eqnarray}\label{xi}\int_{z_i}^{z_{i+1}}[y_{n-1}^2(x,\tau)-y_n(x,\tau)^2]dx
=\frac{-\tau}{b}\ll x(\tau)\gg_i\end{eqnarray} and
\begin{eqnarray}\label{xii}\int_{z_i}^{z_{i+1}}x^2[y_{n-1}^2(x,\tau)-y_n(x,\tau)^2]dx
=\frac{\tau-4b}{5\tau}\ll x(\tau)\gg_i.\end{eqnarray} First of all
notice that $\ll x(\tau)\gg_i\neq0$ for all $\tau>0$. For, if for
some t, $\ll x(\tau)\gg_i=0$ then from \eqref{xi} and \eqref{xii},
we obtain
$$\int_{z_i}^{z_{i+1}}(Ax^2+Bx+c)[y_{n-1}^2(x,\tau)-y_n(x,\tau)^2]=0,$$
where $A$, $B$ and $C$ are arbitrary constants, which is impossible
because there are a points
$z_i(\tau)<x_{2i+1}(\tau)<x_{2i+2}(\tau)<z_{i+1}(\tau)$ such that
\begin{eqnarray}\label{yny}  \left\{
\begin{array}{ll} y^2_n(x,\tau)>y^2_{n-1}(x,\tau),\ \ \ x\in(z_i(\tau),x_{2i+1}(\tau))\cup(x_{2i+2}
(\tau),z_{i+1}(\tau)),\\~~\\
 y^2_n(x,\tau)<y^2_{n-1}(x,\tau),\ \ \ x\in(x_{2i+1}(\tau),x_{2i+2}(\tau)).
 \end{array}
 \right.
\end{eqnarray}
therefore, $\ll x(\tau)\gg_i\neq0$ for all $\tau>0$. Then, according
to the continuity of\\ $\ll x(\tau)\gg_i$, we either have $\ll
x(\tau)\gg_i<0$ or $\ll x(\tau)\gg_i>0.$ We assume the contrary that
$\ll x(\tau)\gg_i<0.$ Then, from this together with \eqref{xi} and
\eqref{xii}, we obtain
\begin{eqnarray}\label{ABC}\int_{z_i}^{z_{i+1}}(Ax^2+Bx+c)[y_{n-1}^2(x,\tau)-y_n(x,\tau)^2]<0,\end{eqnarray}
for all $\tau<4b$. But if we choose $A<0$ , $B>0$ and $C<0$ we get
$$Ax^2+Bx+c=-(x-x_{2i+1}(\tau))(x-x_{2i+12}(\tau)).$$
by \eqref{yny}, we find that \eqref{ABC} is positive. It is a
contradiction with the hypothesis $\ll x(\tau)\gg_i<0.$
\end{proof}
\begin{lemma}\label{theo13}
Let $\hat{\rho}>0$ be function continuous on $[0,1]$ such that
$\hat{\rho}(x)=a_ix+b$ for $x\in[z_i,z_{i+1}]$ . Then the
eigenvalues of Problem \eqref{1.1}-\eqref{1.2} with
$\rho=\hat{\rho},$ satisfy
   \begin{eqnarray}\label{12}
 \frac{\lambda_{n}(\hat{\rho})}{\lambda_{m}(\hat{\rho})}\geq (\frac{n}{m})^2.
 \end{eqnarray}
 Equality holds iff $\hat{\rho}$ is constant in $[0,1]$.
 \end{lemma}

\begin{proof} According to Lemma \ref{lm},
\begin{eqnarray*}\frac{d}{d\tau}\Big[\frac{\lambda_n(\tau)}{\lambda_{n-1}(\tau)}\Big]=
\frac{\lambda_n(\tau)}{\lambda_{n-1}(\tau)}
\int_0^1x(y_{n-1}^2(x,\tau)-y_n^2(x,\tau))dx\geq0.\end{eqnarray*}
Then,
$$\frac{\lambda_n(\tau)}{\lambda_{n-1}(\tau)}\geq
\frac{\lambda_n(0)}{\lambda_{n-1}(0)}=\big(\frac{n}{n-1}\big)^2.$$
Then $$\frac{\lambda_n (\hat{\rho})}{\lambda_{m}(\hat{\rho}) }\geq \
\big(\frac{n}{m}\big)^2.$$
\end{proof}

We are now ready to prove Theorem \ref{theo11}.

\begin{proof}
 We define $\rho(x,\tau)=\tau\rho(x)+(1-\tau)\hat{\rho}(x),$ then
from \eqref{Ra}
\begin{eqnarray*}&&\frac{d}{d\tau}\Big[\frac{\lambda_n(\tau)}{\lambda_{n-1}(\tau)}\Big]=
\frac{\lambda_n(\tau)}{\lambda_{n-1}(\tau)}
\int_0^1[\rho(x)-\hat{\rho}(x)](y_{n-1}^2(x,\tau)-y_n^2(x,\tau))dx\geq0\cr&=&
\frac{\lambda_n(\tau)}{\lambda_{n-1}(\tau)}
\sum_{i=0}^{n}\int_{z_i}^{z_{i+1}}[\rho(x)-\hat{\rho}(x)](y_{n-1}^2(x,\tau)-y_n^2(x,\tau))dx
.\end{eqnarray*}We notice that
$$\int_{z_i}^{z_{i+1}}[\rho(x)-\hat{\rho}(x)](y_{n-1}^2(x,\tau)-y_n^2(x,\tau))dx\geq0.$$
It then follows that,
$$\frac{d}{d\tau}\Big[\frac{\lambda_n(\tau)}{\lambda_{n-1}(\tau)}\Big]\geq0.$$
Thus, by the continuity of eigenvalues, we obtain
$$\frac{\lambda_n(\rho)}{\lambda_{n-1}(\rho)}=\frac{\lambda_n(1)}{\lambda_{n-1}(1)}\geq
\frac{\lambda_n(0)}{\lambda_{n-1}(0)}=\frac{\lambda_n(\hat{\rho})}{\lambda_{n-1}(\hat{\rho})}
.$$ And hence
$$\frac{\lambda_n(\rho)}{\lambda_{m}(\rho)}\geq
\big(\frac{n}{m}\big)^2 .$$ According to Lemma \eqref{theo13},
equality holds, if $\rho=\hat{\rho}=cte$.
\end{proof}
\begin{corollary}\label{theo112}
Let $\rho$ a concave density on $[0,1]$. Then the eigenvalues of the
Dirichlet problem \eqref{1.1}-\eqref{1.2} satisfy
   \begin{equation}\label{RA}
  \lambda_{n}-\lambda_{m}\geq\Big((\frac{n}{m})^2-1\Big)\frac{(m\pi)^2}{\rho_M},
   \end{equation} where $(\rho)_M=\max_{x\in[0,1]}\rho(x).$\\
Equality if and only if $\rho$ is constant.
 \end{corollary}
\section{Eigenvalue ratios for Sturm-Liouville problems with $q\equiv0$.}
In this section, we derive the more general bounds on eigenvalue
ratios that can be obtained in the absence of the potential $q$.
\begin{theorem}\label{theo13} Consider the regular Sturm-Liouville
problem $-(p(x)y^{\prime})^ {\prime }=\lambda \rho (x)y$ with with
Dirichlet boundary conditions \eqref{1.2}. If $p\rho$ a concave
function on $[0,1]$ then
   \begin{eqnarray}\label{12}
 \frac{\lambda_{n}}{\lambda_{m}}\geq (\frac{n}{m})^2.
 \end{eqnarray}
 Equality holds iff $p\rho$ is constant in $[0,1]$.
 \end{theorem}

\begin{proof} By use the Legendre substitution \cite[pp.
227-228]{L}
\begin{eqnarray}\label{SB}
t(x)=\frac{1}{\sigma}\int_{0}^{x}\frac{1}{p(z)}dz,~~
\sigma=\int_{0}^{1}\frac{1}{p(z)}dz,\end{eqnarray} Equation
\eqref{1.1S} can be rewritten in the string equation
$$-\ddot{y}=\lambda \sigma^2\tilde{p}(t)\tilde{\rho}(t)y,$$ where
$\tilde{p}(t)=p(x)$ and $\tilde{\rho}(t)=\rho(x)$. Thus the estimate
\eqref{12} is direct consequence of Theorem \ref{theo11}.
\end{proof}
\begin{corollary}\label{theo13} Consider the regular Sturm-Liouville
problem $-(p(x)y^{\prime})^ {\prime }=\lambda \rho (x)y$ with with
Dirichlet boundary conditions \eqref{1.2}. If $p\rho$ a concave
function on $[0,1]$ then
   \begin{equation}\label{RA}
  \lambda_{n}-\lambda_{m}\geq\Big((\frac{n}{m})^2-1\Big)\frac{(m\pi)^2}{(p\rho)_M},
   \end{equation} where $(p\rho)_M=\max_{x\in[0,1]}(p(x)\rho(x)).$\\
 Equality holds iff $p\rho$ is constant in $[0,1]$.
 \end{corollary}

~~~~~~~~~~~\\ {\bf{Acknowledgement}}.{ Research supported by Partial
differential equations laboratory $(LR03ES04),$ at the Faculty of
Sciences of Tunis, University of Tunis El Manar, $2092,$ Tunis,
Tunisia}

\vskip 3cm

\begin{thebibliography}{99}
\addcontentsline{toc}{chapter}{Bibliographie}
\bibitem{1} M. S. Ashbaugh and R. D. Benguria, Optimal bounds for ratios of
eigenvalues of one dimensional Schr\"{o}dinger operators with
Dirichlet boundary conditions and positive potentials, Comm. Math.
Phys., $124,$  $(1989),$ $403-415$.

\bibitem{2} M. Ashbaugh and R. Benguria, Optimal lower bound for the
gap between the first two eigenvalues of one-dimensional
Schr\"{o}dinger operators with symmetric single-well potentials,
Proc. Amer. Math. Soc., $105,$ $(1989),$ $419-424$.

\bibitem{2'} M. S. Ashbaugh and R. D. Benguria, Eigenvalue ratios for Sturm-Liouville
 operators, J. Differential Equations, $103,$ $(1993),$ $205-219.$


\bibitem{JJ} J. Ben Amara and Jihed Hedhly, Eigenvalue ratios for
Schr\"{o}dinger operators with indefinite potentials, Applied
Mathematics Letters, $76,$ $(2018),$ $96-102.$

\bibitem {J} J. Hedhly,  Eigenvalue Ratios for vibrating string equations
 with single-well densities, J. Differential Equations, $2021.$

\bibitem{AC1} A. Constantin, A general-weighted Sturm-Liouville
problem, Annali della Scuola Normale Superiore di Pisa, Classe di
Scienze,  $24,$ $(1997),$ $ 767-782.$

\bibitem{AC2} A. Constantin, On the Inverse Spectral Problem for the
 Camassa-Holm Equation, journal of functional analysis, $155,$ $(1998),$
 $352-363$.

\bibitem{H99} M-J. Huang, On The Eigenvalue Ratio For Vibtrating
Strings, Proc. Amer. Math. Soc., $ 127,$ $(2006),$ $1805-1813.$

\bibitem{H2016} M. J. Huang, The eigenvalue ratio for a class of
densities, J. Math. Anal. Appl.,  $435,$ $(2016),$ $944-954.$

\bibitem{4} Y. L. Huang and C. K. Law, Eigenvalue ratios for the regular
Sturm-Liouville system, Proc. Amer. Math. Soc., $ 124$, $(1996),$
$1427-1436.$

\bibitem{MH} M. Horv\'{a}th, on the first two eigenvalues of
Stum-Liouville operators, Proc. Amer. Math. Soc., $131,$ $(2002),$
$1215-1224.$
\bibitem{5'} M. Horv\'{a}th and M. Kiss, A bound for ratios
of eigenvalues of Schr\"{o}dinger operators with single-well
potentials, Proc. Amer. Math. Soc., $134,$ $( 2005 ),$ $1425-1434.$
.

\bibitem{KL} J. B. Keller, The minimum ratio of two eigenvalues,
 SIAM J. Appl. Math., $ 31$, $(1976),$ $485-491. $

\bibitem{M'} M. Kiss, Eigenvalue ratios of vibrating strings, Acta Math.
Hungar., $110,$ $2006,$ $253-259.$


\bibitem{L} W. Leighton, Ordinary Differential Equations. 3rd ed.
Wadsworth, Belmont. CA., $1970$.

\bibitem{5} B.M. Levitan and I. S. Sargsyan,  Introduction to spectral
theory: Selfadjoint Ordinary Differential Operators, American
Mathematical Society,
 Translation of Mathematical Monographs, $39,$ $(1975)$.
\end{thebibliography}
\end{document}